\documentclass{amsart}
\usepackage{amsmath}
\usepackage{amsfonts}

\newtheorem{theorem}{Theorem}
\theoremstyle{plain}

\newtheorem{definition}{Definition}
\newtheorem{example}{Example}

\newtheorem{lemma}{Lemma}

\newtheorem{problem}{Problem}

\numberwithin{equation}{section}
\input{tcilatex}

\begin{document}
\title[A new fastest iteration method]{A Picard-S hybrid type iteration
method for solving a differential equation with retarded argument}
\author{Faik G\"{U}RSOY}
\address{Faculty of Arts and Sciences, Department of Mathematics, Adiyaman
University, 02040 Adiyaman, Turkey.}
\email{faikgursoy02@hotmail.com}
\author{Vatan KARAKAYA}
\address{Department of Mathematical Engineering, Yildiz Technical
University, Davutpasa Campus, Esenler, 34210 Istanbul, Turkey}
\email{vkkaya@yahoo.com; vkkaya@yildiz.edu.tr}
\urladdr{http://www.yarbis1.yildiz.edu.tr/vkkaya}

\begin{abstract}
We introduce a new iteration method called Picard-S iteration. We show that
the Picard-S iteration method can be used to approximate the fixed point of
contraction mappings. Also, we show that our new iteration method is
equivalent and converges faster than CR iteration method for the
aforementioned class of mappings. Furthermore, by providing an example, it
is shown that the Picard-S iteration method converges faster than all
Picard, Mann, Ishikawa, Noor, SP, CR, S and some other iteration methods in
the existing literature. A data dependence result is proven for fixed point
of contraction mappings with the help of the new iteration method. Finally,
we show that the Picard-S iteration method can be used to solve differential
equations with retarded argument.
\end{abstract}

\maketitle

\section{Introduction}

Fixed point theory has been appeared as one of the most powerful and
substantial theoretical tools of mathematics. This theory has a long history
and has been studied intensively by many researchers in various aspects. The
main objective of studies in the fixed point theory is to find solutions for
the following equation which is commonly known as fixed point equation:%
\begin{equation}
Tx=x\text{,}  \label{eqn1}
\end{equation}%
where $T$ is a self-map of an ambient space $X$ and $x\in X$. A wide variety
of problems arise in all areas of physical, chemical and biological
sciences, engineering, economics, and management can be modelled by linear
or nonlinear equations of form 
\begin{equation}
Fx=0\text{,}  \label{eqn2}
\end{equation}%
where $F$ is a linear or nonlinear operator. Equations of form (1.2) can be
easily reformulated as a fixed point equation of form (1.1). Since equation
(1.1) has the same solution as the original equation (1.1), finding
solutions of equation (1.1) leads to solutions of equation (1.2). To solve
equations given by (1.1), two types of methods are normally used: direct
methods and iterative methods. Due to various reasons, direct methods can be
impractical or fail in solving equations (1.1), and thus iterative methods
become a viable alternative. For this reason, the iterative approximation of
fixed points has become one of the major and basic tools in the theory of
equations. Consequently, the literature of this highly dynamic research area
abounds many iterative methods that have been introduced and developed by a
wide audience of researchers to serve various purposes, viz., \cite{GK,
KMSP, DataF, Kkaya, Khan, Ola, SahuPet, SolOt1, SolOt2} among others.

We begin our exposition with an overview of various iterative methods.

We will denote the set of all positive integers including zero by $%
\mathbb{N}
$ over the course of this paper. Let $D$ be a nonempty convex subset of a
Banach space $B$, and $T$ a self map of $D$. A point $x_{\ast }\in D$
satisfying $Tx=x$ is called fixed point of $T$, and the set of all fixed
point of $T$ denoted by $F_{T}$. Let $\left\{ \eta _{n}^{i}\right\}
_{n=0}^{\infty }$, $i\in \left\{ 0,1,2\right\} $ be real sequences in $\left[
0,1\right] $ satisfying certain control condition(s).

The most popular and simplest iteration method is formulated by%
\begin{equation}
\left\{ 
\begin{array}{c}
p_{0}\in D\text{, \ \ \ \ \ \ \ \ \ \ \ \ \ \ \ } \\ 
p_{n+1}=Tp_{n}\text{, }n\in 
\mathbb{N}
\text{,}%
\end{array}%
\right.  \label{eqn3}
\end{equation}%
and is known as Picard iteration method \cite{Picard}, which is commonly
used to approximate fixed point of contraction mappings satisfying%
\begin{equation}
\left\Vert Tx-Ty\right\Vert \leq \delta \left\Vert x-y\right\Vert \text{, }%
\delta \in \left( 0,1\right) \text{, for all }x,y\in B\text{.}  \label{eqn4}
\end{equation}%
The following iteration methods are referred to as Mann \cite{Mann},
Ishikawa \cite{Ishikawa}, Noor \cite{Noor}, SP \cite{SP}, S \cite{S, Sahu}
and CR \cite{CR} iteration methods, respectively:%
\begin{equation}
\left\{ 
\begin{array}{c}
\nu _{0}\in D\text{,\ \ \ \ \ \ \ \ \ \ \ \ \ \ \ \ \ \ \ \ \ \ \ \ \ \ \ \
\ \ \ \ \ \ \ \ \ \ } \\ 
\nu _{n+1}=\left( 1-\eta _{n}^{0}\right) \nu _{n}+\eta _{n}^{0}T\nu _{n}%
\text{, }n\in 
\mathbb{N}
\text{,}%
\end{array}%
\right.  \label{eqn5}
\end{equation}%
\begin{equation}
\left\{ 
\begin{array}{c}
\upsilon _{0}\in D\text{, \ \ \ \ \ \ \ \ \ \ \ \ \ \ \ \ \ \ \ \ \ \ \ \ \
\ \ \ \ \ \ \ \ \ \ \ \ \ \ \ } \\ 
\upsilon _{n+1}=\left( 1-\eta _{n}^{0}\right) \upsilon _{n}+\eta
_{n}^{0}Tw_{n}\text{, \ \ \ \ \ \ \ \ \ \ \ } \\ 
w_{n}=\left( 1-\eta _{n}^{1}\right) \upsilon _{n}+\eta _{n}^{1}T\upsilon _{n}%
\text{, }n\in 
\mathbb{N}
\text{,}%
\end{array}%
\right.  \label{eqn6}
\end{equation}%
\begin{equation}
\left\{ 
\begin{array}{c}
\omega _{0}\in D\text{, \ \ \ \ \ \ \ \ \ \ \ \ \ \ \ \ \ \ \ \ \ \ \ \ \ \
\ \ \ \ \ \ \ \ \ \ \ \ \ \ } \\ 
\omega _{n+1}=\left( 1-\eta _{n}^{0}\right) \omega _{n}+\eta _{n}^{0}T\varpi
_{n}\text{, \ \ \ \ \ \ \ \ \ \ \ } \\ 
\varpi _{n}=\left( 1-\eta _{n}^{1}\right) \omega _{n}+\eta _{n}^{1}T\rho _{n}%
\text{, \ \ \ \ \ \ \ \ \ } \\ 
\rho _{n}=\left( 1-\eta _{n}^{2}\right) \omega _{n}+\eta _{n}^{2}T\omega _{n}%
\text{,~}n\in 
\mathbb{N}
\text{,}%
\end{array}%
\right.  \label{eqn7}
\end{equation}%
\begin{equation}
\left\{ 
\begin{array}{c}
q_{0}\in D\text{,\ \ \ \ \ \ \ \ \ \ \ \ \ \ \ \ \ \ \ \ \ \ \ \ \ \ \ \ \ \
\ \ \ \ \ \ \ \ \ \ \ } \\ 
q_{n+1}=\left( 1-\eta _{n}^{1}\right) r_{n}+\eta _{n}^{1}Tr_{n}\text{,\ \ \
\ \ \ \ \ \ \ \ \ \ } \\ 
r_{n}=\left( 1-\eta _{n}^{2}\right) s_{n}+\eta _{n}^{2}Ts_{n}\text{,\ \ \ \
\ \ \ \ \ \ } \\ 
s_{n}=\left( 1-\eta _{n}^{3}\right) q_{n}+\eta _{n}^{3}Tq_{n}\text{, }n\in 
\mathbb{N}
\text{,}%
\end{array}%
\right.  \label{eqn8}
\end{equation}%
\begin{equation}
\left\{ 
\begin{array}{c}
t_{0}\in D\text{, \ \ \ \ \ \ \ \ \ \ \ \ \ \ \ \ \ \ \ \ \ \ \ \ \ \ \ \ \
\ \ \ \ \ \ \ \ \ \ } \\ 
t_{n+1}=\left( 1-\eta _{n}^{0}\right) Tt_{n}+\eta _{n}^{0}Tu_{n}\text{,\ \ \
\ \ \ \ \ \ \ \ } \\ 
u_{n}=\left( 1-\eta _{n}^{1}\right) t_{n}+\eta _{n}^{1}Tt_{n}\text{, }n\in 
\mathbb{N}
\text{, \ }%
\end{array}%
\right.  \label{eqn9}
\end{equation}%
\begin{equation}
\left\{ 
\begin{array}{c}
u_{0}\in D\text{, \ \ \ \ \ \ \ \ \ \ \ \ \ \ \ \ \ \ \ \ \ \ \ \ \ \ \ \ \
\ \ \ \ \ \ \ \ \ \ \ } \\ 
u_{n+1}=\left( 1-\eta _{n}^{0}\right) v_{n}+\eta _{n}^{0}Tv_{n}\text{, \ \ \
\ \ \ \ \ \ \ \ \ \ } \\ 
v_{n}=\left( 1-\eta _{n}^{1}\right) Tu_{n}+\eta _{n}^{1}Tw_{n}\text{, \ \ \
\ \ \ } \\ 
w_{n}=\left( 1-\eta _{n}^{2}\right) u_{n}+\eta _{n}^{2}Tu_{n}\text{, }n\in 
\mathbb{N}
\text{.}%
\end{array}%
\right.  \label{eqn10}
\end{equation}%
Studying the convergence of fixed point iteration methods is of utmost
importance from various aspects and thus, in recent years, much attention
has been given to study of convergence of various iterative methods for
different classes of operators, e.g. \cite{Kirk, Krasnosel'skii, KM-I,
Olatinwo, GenKrasnosel'skii}. In many cases, there can be more than one
iteration method to approximate fixed points of a particular mapping, e.g. 
\cite{SS, Karakaya, Rhoades, Multistep, Rhoades1, Zxue}. In such cases, one
must make a choice among some given iteration methods by taking into account
some important criteria. For example, there are two main criteria such as
the speed of convergence and simplicity which make an iteration method more
effective than the others. In such cases, the following problems arise
naturally: Which one of these iteration methods converges faster to the
fixed point in question? and How do we compare the speed of these iteration
methods that converge to the same fixed point? There are only a few attempts
to solve such an important numerical problem, see \cite{Babu, Berinde1,
Berinde2, KN, Hussain1, Hussain, Karakaya, Suantai, Popescu, ZXUE, Yuan}.

The following definitions about the rate of convergence are due to Berinde 
\cite{Berinde}.

\begin{definition}
\cite{Berinde} Let $\left\{ a_{n}\right\} _{n=0}^{\infty }$ and $\left\{
b_{n}\right\} _{n=0}^{\infty }$ be two sequences of real numbers with limits 
$a$ and $b$, respectively. Assume that there exists%
\begin{equation}
\underset{n\rightarrow \infty }{\lim }\frac{\left\vert a_{n}-a\right\vert }{%
\left\vert b_{n}-b\right\vert }=l\text{.}  \label{eqn11}
\end{equation}

(i) If $l=0$, then we say that $\left\{ a_{n}\right\} _{n=0}^{\infty }$
converges faster to $a$ than $\left\{ b_{n}\right\} _{n=0}^{\infty }$ to $b$.

(ii) If $0<l<\infty $, then we say that $\left\{ a_{n}\right\}
_{n=0}^{\infty }$ and $\left\{ b_{n}\right\} _{n=0}^{\infty }$ have the same
rate of convergence.
\end{definition}

\begin{definition}
Suppose that for two fixed point iteration processes $\left\{ u_{n}\right\}
_{n=0}^{\infty }$ and $\left\{ v_{n}\right\} _{n=0}^{\infty }$ both
converging to the same fixed point $p$, the following error estimates%
\begin{equation}
\left\Vert u_{n}-p\right\Vert \leq a_{n}\text{ for all }n\in 
\mathbb{N}
\text{,}  \label{eqn12}
\end{equation}%
\begin{equation}
\left\Vert v_{n}-p\right\Vert \leq b_{n}\text{ for all }n\in 
\mathbb{N}
\text{,}  \label{eqn13}
\end{equation}%
are available where $\left\{ a_{n}\right\} _{n=0}^{\infty }$ and $\left\{
b_{n}\right\} _{n=0}^{\infty }$ are two sequences of positive numbers
(converging to zero). If $\left\{ a_{n}\right\} _{n=0}^{\infty }$ converges
faster than $\left\{ b_{n}\right\} _{n=0}^{\infty }$, then $\left\{
u_{n}\right\} _{n=0}^{\infty }$ converges faster than $\left\{ v_{n}\right\}
_{n=0}^{\infty }$ to $p$.
\end{definition}

In the sequel, whenever we talk about the rate of convergence, we refer to
the definitions given above.

In 2007, Agarwal et al. \cite{S} addressed the following question: \textit{%
Is it possible to develop an iteration process whose rate of convergence is
faster than the Picard iteration (1.3)?}

They answered this problem by introducing an S-iteration method defined by
(1.9). It was shown in \cite{S} that S-iteration method (1.9) converges at a
rate same as that of Picard iteration method (1.3) and faster than Mann
iteration method (1.5) for the class of contraction mappings satisfying
(1.4). In 2009, Sahu \cite{Sahu} was interested in the same problem and
proved both theoretically and numerically that S-iteration method (1.9)
converges at a rate faster than both Picard iteration method (1.3) and Mann
iteration method (1.5) for the class of contraction mappings (1.4). A recent
study by Chugh, et al. \cite{CR} shows that CR iterative method (1.10)
converges faster than the Picard (1.3), Mann (1.5), Ishikawa (1.6), Noor
(1.7), SP (1.8) and S (1.9) iterative methods for a particular class of
quasi-contractive operators which include the aforementioned class of
contraction operators.

Inspired by the works mentioned above, we propose the following problem:

\begin{problem}
Is it possible to develop an iteration process whose rate of convergence is
even faster than the iteration (1.10)?
\end{problem}

To answer this problem, we introduce the following iteration method called
Picard-S iteration:

\begin{equation}
\left\{ 
\begin{array}{c}
x_{0}\in D\text{, \ \ \ \ \ \ \ \ \ \ \ \ \ \ \ \ \ \ \ \ \ \ \ \ \ \ \ \ \
\ \ \ \ \ \ \ \ \ \ \ } \\ 
x_{n+1}=Ty_{n}\text{, \ \ \ \ \ \ \ \ \ \ \ \ \ \ \ \ \ \ \ \ \ \ \ \ \ \ \
\ \ \ \ \ \ \ } \\ 
y_{n}=\left( 1-\eta _{n}^{1}\right) Tx_{n}+\eta _{n}^{1}Tz_{n}\text{, \ \ \
\ \ \ } \\ 
z_{n}=\left( 1-\eta _{n}^{2}\right) x_{n}+\eta _{n}^{2}Tx_{n}\text{, }n\in 
\mathbb{N}
\text{,}%
\end{array}%
\right.  \label{eqn14}
\end{equation}%
In this paper, we show that the Picard-S iteration method can be used to
approximate fixed point of contraction mappings. Also, we show that our new
iteration method is equivalent and converges faster than CR iteration method
for the aforementioned class of mappings. Furthermore, by providing an
example, when applied to contraction mappings it is shown that the Picard-S
iteration method converges faster than CR iteration method and hence also
faster than all Picard, Mann, Ishikawa, Noor, SP, S and some other iteration
methods in the existing literature. A data dependence result is proven for
fixed point of contraction mappings with the help of the new iteration
method. Finally, we show that the Picard-S iteration method can be used as
an effective method to solve differential equations with retarded argument.

In order to obtain our main results we need following definition and lemmas.

\begin{definition}
\cite{Vasile} Let $T$,$\widetilde{T}:B\rightarrow B$ be two operators. We
say that $\widetilde{T}$ is an approximate operator of $T$ if for all $x\in
B $ and for a fixed $\varepsilon >0$ we have 
\begin{equation}
\left\Vert Tx-\widetilde{T}x\right\Vert \leq \varepsilon .  \label{eqn15}
\end{equation}
\end{definition}

\begin{lemma}
\cite{Weng}Let $\left\{ \beta _{n}\right\} _{n=0}^{\infty }$ and $\left\{
\rho _{n}\right\} _{n=0}^{\infty }$ be nonnegative real sequences satisfying
the following inequality:%
\begin{equation}
\beta _{n+1}\leq \left( 1-\lambda _{n}\right) \beta _{n}+\rho _{n}\text{,}
\label{eqn16}
\end{equation}%
where $\lambda _{n}\in \left( 0,1\right) $, for all $n\geq n_{0}$, $%
\dsum\nolimits_{n=1}^{\infty }\lambda _{n}=\infty $, and $\frac{\rho _{n}}{%
\lambda _{n}}\rightarrow 0$ as $n\rightarrow \infty $. Then $%
\lim_{n\rightarrow \infty }\beta _{n}=0$.
\end{lemma}

\begin{lemma}
\cite{Data Is 2} Let $\left\{ \beta _{n}\right\} _{n=0}^{\infty }$ be a
nonnegative sequence for which one assumes there exists $n_{0}\in 
\mathbb{N}
$, such that for all $n\geq n_{0}$ one has satisfied the inequality 
\begin{equation}
\beta _{n+1}\leq \left( 1-\mu _{n}\right) \beta _{n}+\mu _{n}\gamma _{n}%
\text{,}  \label{eqn17}
\end{equation}%
where $\mu _{n}\in \left( 0,1\right) ,$ for all $n\in 
\mathbb{N}
$, $\sum\limits_{n=0}^{\infty }\mu _{n}=\infty $ and $\eta _{n}\geq 0$, $%
\forall n\in 
\mathbb{N}
$. Then the following inequality holds 
\begin{equation}
0\leq \lim \sup_{n\rightarrow \infty }\beta _{n}\leq \lim \sup_{n\rightarrow
\infty }\gamma _{n}.  \label{eqn18}
\end{equation}
\end{lemma}

\section{Main Results}

\textbf{Convergence Analysis}

\begin{theorem}
Let $D$ be a nonempty closed convex subset of a Banach space $B$ and $%
T:D\rightarrow D$ a contraction map satisfying condition (1.4). Let $\left\{
x_{n}\right\} _{n=0}^{\infty }$ be an iterative sequence generated by (1.14)
with real sequences $\left\{ \eta _{n}^{i}\right\} _{n=0}^{\infty } $, $i\in
\left\{ 1,2\right\} $ in $\left[ 0,1\right] $ satisfying $\sum_{k=0}^{n}\eta
_{k}^{1}\eta _{k}^{2}=\infty $. Then $\ \left\{ x_{n}\right\} _{n=0}^{\infty
}$ converges to a unique fixed point of $T$, say $x_{\ast }$.
\end{theorem}

\begin{proof}
The well-known Picard-Banach theorem guarantees the existence and uniqueness
of $x_{\ast }$. We will show that $x_{n}\rightarrow x_{\ast }$ as $%
n\rightarrow \infty $. From (1.4) and (1.14) we have%
\begin{eqnarray}
\left\Vert z_{n}-x_{\ast }\right\Vert &=&\left\Vert \left( 1-\eta
_{n}^{2}\right) x_{n}+\eta _{n}^{2}Tx_{n}-\left( 1-\eta _{n}^{2}+\eta
_{n}^{2}\right) x_{\ast }\right\Vert  \notag \\
&\leq &\left( 1-\eta _{n}^{2}\right) \left\Vert x_{n}-x_{\ast }\right\Vert
+\eta _{n}^{2}\left\Vert Tx_{n}-Tx_{\ast }\right\Vert  \notag \\
&\leq &\left( 1-\eta _{n}^{2}\right) \left\Vert x_{n}-x_{\ast }\right\Vert
+\eta _{n}^{2}\delta \left\Vert x_{n}-x_{\ast }\right\Vert  \notag \\
&=&\left[ 1-\eta _{n}^{2}\left( 1-\delta \right) \right] \left\Vert
x_{n}-x_{\ast }\right\Vert \text{,}  \label{eqn19}
\end{eqnarray}%
\begin{eqnarray}
\left\Vert y_{n}-x_{\ast }\right\Vert &\leq &\left( 1-\eta _{n}^{1}\right)
\left\Vert Tx_{n}-Tx_{\ast }\right\Vert +\eta _{n}^{1}\left\Vert
Tz_{n}-Tx_{\ast }\right\Vert  \notag \\
&\leq &\left( 1-\eta _{n}^{1}\right) \delta \left\Vert x_{n}-x_{\ast
}\right\Vert +\eta _{n}^{1}\delta \left\Vert z_{n}-x_{\ast }\right\Vert 
\notag \\
&\leq &\delta \left\{ 1-\eta _{n}^{1}\eta _{n}^{2}\left( 1-\delta \right)
\right\} \left\Vert x_{n}-x_{\ast }\right\Vert \text{,}  \label{eqn20}
\end{eqnarray}%
\begin{eqnarray}
\left\Vert x_{n+1}-x_{\ast }\right\Vert &\leq &\left\Vert Ty_{n}-Tx_{\ast
}\right\Vert \leq \delta \left\Vert y_{n}-x_{\ast }\right\Vert  \notag \\
&\leq &\delta ^{2}\left\{ 1-\eta _{n}^{1}\eta _{n}^{2}\left( 1-\delta
\right) \right\} \left\Vert x_{n}-x_{\ast }\right\Vert \text{.}
\label{eqn21}
\end{eqnarray}%
Repetition of above processes gives the following inequalities%
\begin{equation}
\left\{ 
\begin{array}{c}
\left\Vert x_{n+1}-x_{\ast }\right\Vert \leq \delta ^{2}\left\{ 1-\eta
_{n}^{1}\eta _{n}^{2}\left( 1-\delta \right) \right\} \left\Vert
x_{n}-x_{\ast }\right\Vert \text{, \ \ \ \ \ \ \ \ \ \ \ } \\ 
\left\Vert x_{n}-x_{\ast }\right\Vert \leq \delta ^{2}\left\{ 1-\eta
_{n-1}^{1}\eta _{n-1}^{2}\left( 1-\delta \right) \right\} \left\Vert
x_{n-1}-x_{\ast }\right\Vert \text{,} \\ 
\left\Vert x_{n-1}-x_{\ast }\right\Vert \leq \delta ^{2}\left\{ 1-\eta
_{n-2}^{1}\eta _{n-2}^{2}\left( 1-\delta \right) \right\} \left\Vert
x_{n-2}-x_{\ast }\right\Vert \text{, \ \ } \\ 
\vdots \text{ \ \ \ \ \ \ \ \ \ \ \ \ \ \ \ \ \ \ \ \ \ \ \ \ \ \ \ \ \ \ \
\ \ \ \ \ \ } \\ 
\left\Vert x_{1}-x_{\ast }\right\Vert \leq \delta ^{2}\left\{ 1-\eta
_{0}^{1}\eta _{0}^{2}\left( 1-\delta \right) \right\} \left\Vert
x_{0}-x_{\ast }\right\Vert \text{. \ \ \ \ \ \ \ \ }%
\end{array}%
\right.  \label{eqn22}
\end{equation}%
From inequalities given by (2.4), we derive%
\begin{equation}
\left\Vert x_{n+1}-x_{\ast }\right\Vert \leq \left\Vert x_{0}-x_{\ast
}\right\Vert \delta ^{2\left( n+1\right) }\prod\limits_{k=0}^{n}\left\{
1-\eta _{k}^{1}\eta _{k}^{2}\left( 1-\delta \right) \right\} \text{.}
\label{eqn23}
\end{equation}%
Since $\delta \in \left( 0,1\right) $ and $\eta _{n}^{i}\in \left[ 0,1\right]
$, for all $n\in 
\mathbb{N}
$ and for each $i\in \left\{ 1,2\right\} $%
\begin{equation}
1-\eta _{n}^{1}\eta _{n}^{2}\left( 1-\delta \right) <1  \label{eqn24}
\end{equation}%
It is well-known from the classical analysis that $1-x\leq e^{-x}$ for all $%
x\in \left[ 0,1\right] $. Taking into account these facts together with
(2.5), we obtain%
\begin{equation}
\left\Vert x_{n+1}-x_{\ast }\right\Vert \leq \left\Vert x_{0}-x_{\ast
}\right\Vert \delta ^{2\left( n+1\right) }e^{-\left( 1-\delta \right)
\sum_{k=0}^{n}\eta _{k}^{1}\eta _{k}^{2}}\text{.}  \label{eqn25}
\end{equation}%
Taking the limit of both sides of inequality (2.7) yields $%
\lim_{n\rightarrow \infty }\left\Vert x_{n}-x_{\ast }\right\Vert =0$, i.e. $%
x_{n}\rightarrow x_{\ast }$ as $n\rightarrow \infty $.
\end{proof}

\begin{theorem}
Let $D$, $B$ and $T$ with fixed point $x_{\ast }$ be as in Theorem 1. Let $%
\{u_{n}\}_{n=0}^{\infty }$, $\{x_{n}\}_{n=0}^{\infty }$be two iterative
sequences defined by CR (1.10) and Picard-S (1.14) iteration methods,
respectively, with real sequences $\left\{ \eta _{n}^{i}\right\}
_{n=0}^{\infty }$, $i\in \left\{ 1,2,3\right\} $ in $\left[ 0,1\right] $
satisfying $\sum_{k=0}^{n}\eta _{k}^{1}\eta _{k}^{2}=\infty $. Then the
following are equivalent:

(i) $\lim_{n\rightarrow \infty }\left\Vert x_{n}-x_{\ast }\right\Vert =0$;

(ii) $\lim_{n\rightarrow \infty }\left\Vert u_{n}-x_{\ast }\right\Vert =0$.
\end{theorem}

\begin{proof}
We will prove (i)$\Rightarrow $(ii), that is, if Picard-S iteration method
(1.14) converges to $x_{\ast }$, then CR iteration method (1.10) does too.
Now by using (1.14), (1.10) and condition (1.4), we have%
\begin{eqnarray}
\left\Vert x_{n+1}-u_{n+1}\right\Vert &=&\left\Vert Ty_{n}-\left( 1-\eta
_{n}^{0}\right) v_{n}-\eta _{n}^{0}Tv_{n}\right\Vert  \notag \\
&=&\left\Vert \left( 1-\eta _{n}^{0}+\eta _{n}^{0}\right) Ty_{n}-\left(
1-\eta _{n}^{0}\right) v_{n}-\eta _{n}^{0}Tv_{n}\right\Vert  \notag \\
&\leq &\left( 1-\eta _{n}^{0}\right) \left\Vert Ty_{n}-v_{n}\right\Vert
+\eta _{n}^{0}\left\Vert Ty_{n}-Tv_{n}\right\Vert  \notag \\
&\leq &\left[ 1-\eta _{n}^{0}\left( 1-\delta \right) \right] \left\Vert
y_{n}-v_{n}\right\Vert +\left( 1-\eta _{n}^{0}\right) \left\Vert
y_{n}-Ty_{n}\right\Vert  \label{eqn26}
\end{eqnarray}%
\begin{eqnarray}
\left\Vert y_{n}-v_{n}\right\Vert &=&\left\Vert \left( 1-\eta
_{n}^{1}\right) Tx_{n}+\eta _{n}^{1}Tz_{n}-\left( 1-\eta _{n}^{1}\right)
Tu_{n}-\eta _{n}^{1}Tw_{n}\right\Vert  \notag \\
&\leq &\left( 1-\eta _{n}^{1}\right) \left\Vert Tx_{n}-Tu_{n}\right\Vert
+\eta _{n}^{1}\left\Vert Tz_{n}-Tw_{n}\right\Vert  \notag \\
&\leq &\left( 1-\eta _{n}^{1}\right) \delta \left\Vert
x_{n}-u_{n}\right\Vert +\eta _{n}^{1}\delta \left\Vert z_{n}-w_{n}\right\Vert
\label{eqn27}
\end{eqnarray}%
\begin{eqnarray}
\left\Vert z_{n}-w_{n}\right\Vert &=&\left\Vert \left( 1-\eta
_{n}^{2}\right) x_{n}+\eta _{n}^{2}Tx_{n}-\left( 1-\eta _{n}^{2}\right)
u_{n}-\eta _{n}^{2}Tu_{n}\right\Vert  \notag \\
&\leq &\left( 1-\eta _{n}^{2}\right) \left\Vert x_{n}-u_{n}\right\Vert +\eta
_{n}^{2}\left\Vert Tx_{n}-Tu_{n}\right\Vert  \notag \\
&\leq &\left[ 1-\eta _{n}^{2}\left( 1-\delta \right) \right] \left\Vert
x_{n}-u_{n}\right\Vert  \label{eqn28}
\end{eqnarray}%
Combining (2.8), (2.9), and (2.10)%
\begin{eqnarray}
\left\Vert x_{n+1}-u_{n+1}\right\Vert &\leq &\delta \left[ 1-\eta
_{n}^{0}\left( 1-\delta \right) \right] \left[ 1-\eta _{n}^{1}\eta
_{n}^{2}\left( 1-\delta \right) \right] \left\Vert x_{n}-u_{n}\right\Vert 
\notag \\
&&+\left( 1-\eta _{n}^{0}\right) \left\Vert y_{n}-Ty_{n}\right\Vert
\label{eqn29}
\end{eqnarray}%
Since $\delta \in \left( 0,1\right) $ and $\eta _{n}^{i}\in \left[ 0,1\right]
$, for all $n\in 
\mathbb{N}
$ and for each $i\in \left\{ 0,1,2\right\} $%
\begin{equation}
\delta \left[ 1-\eta _{n}^{0}\left( 1-\delta \right) \right] <1\text{.}
\label{eqn30}
\end{equation}%
An application of the inequalitiy (2.12) to (2.11) yields%
\begin{equation}
\left\Vert x_{n+1}-u_{n+1}\right\Vert \leq \left[ 1-\eta _{n}^{1}\eta
_{n}^{2}\left( 1-\delta \right) \right] \left\Vert x_{n}-u_{n}\right\Vert
+\left( 1-\eta _{n}^{0}\right) \left\Vert y_{n}-Ty_{n}\right\Vert \text{.}
\label{eqn31}
\end{equation}%
Define%
\begin{equation}
\beta _{n}:=\left\Vert x_{n}-u_{n}\right\Vert \text{, }\lambda _{n}:=\eta
_{n}^{1}\eta _{n}^{2}\left( 1-\delta \right) \in \left( 0,1\right) \text{, }%
\rho _{n}:=\left( 1-\eta _{n}^{0}\right) \left\Vert y_{n}-Ty_{n}\right\Vert 
\text{.}  \label{eqn32}
\end{equation}%
Since $\lim_{n\rightarrow \infty }\left\Vert x_{n}-x_{\ast }\right\Vert =0$
and $Tx_{\ast }=x_{\ast }$, $\lim_{n\rightarrow \infty }\left\Vert
y_{n}-Ty_{n}\right\Vert =0$ which implies $\frac{\rho _{n}}{\lambda _{n}}%
\rightarrow 0$ as $n\rightarrow \infty $. Thus all conditions of Lemma 1 are
fulfilled by (2.13), and so $\lim_{n\rightarrow \infty }\left\Vert
x_{n}-u_{n}\right\Vert =0$. Since%
\begin{equation}
\left\Vert u_{n}-x_{\ast }\right\Vert \leq \left\Vert x_{n}-u_{n}\right\Vert
+\left\Vert x_{n}-x_{\ast }\right\Vert \text{,}  \label{eqn33}
\end{equation}%
$\lim_{n\rightarrow \infty }\left\Vert u_{n}-x_{\ast }\right\Vert =0$.

Next we will prove (ii)$\Rightarrow $(i). Using (1.14), (1.10) and condition
(1.4), we have%
\begin{eqnarray}
\left\Vert u_{n+1}-x_{n+1}\right\Vert &=&\left\Vert v_{n}-Ty_{n}-\eta
_{n}^{0}\left( v_{n}-Tv_{n}\right) \right\Vert  \notag \\
&\leq &\left\Vert v_{n}-Ty_{n}\right\Vert +\eta _{n}^{0}\left\Vert
v_{n}-Tv_{n}\right\Vert  \notag \\
&\leq &\left\Vert Tv_{n}-Ty_{n}\right\Vert +\left( 1+\eta _{n}^{0}\right)
\left\Vert v_{n}-Tv_{n}\right\Vert  \notag \\
&\leq &\delta \left\Vert v_{n}-y_{n}\right\Vert +\left( 1+\eta
_{n}^{0}\right) \left\Vert v_{n}-Tv_{n}\right\Vert \text{,}  \label{eqn34}
\end{eqnarray}%
\begin{eqnarray}
\left\Vert v_{n}-y_{n}\right\Vert &\leq &\left( 1-\eta _{n}^{1}\right)
\left\Vert Tu_{n}-Tx_{n}\right\Vert +\eta _{n}^{1}\left\Vert
Tw_{n}-Tz_{n}\right\Vert  \notag \\
&\leq &\delta \left\{ \left( 1-\eta _{n}^{1}\right) \left\Vert
u_{n}-x_{n}\right\Vert +\eta _{n}^{1}\left\Vert w_{n}-z_{n}\right\Vert
\right\} \text{,}  \label{eqn35}
\end{eqnarray}%
\begin{eqnarray}
\left\Vert w_{n}-z_{n}\right\Vert &\leq &\left( 1-\eta _{n}^{2}\right)
\left\Vert u_{n}-x_{n}\right\Vert +\eta _{n}^{2}\left\Vert
Tu_{n}-Tx_{n}\right\Vert  \notag \\
&\leq &\left[ 1-\eta _{n}^{2}\left( 1-\delta \right) \right] \left\Vert
u_{n}-x_{n}\right\Vert \text{.}  \label{eqn36}
\end{eqnarray}%
Combining (2.16), (2.17), and (2.18)%
\begin{eqnarray}
\left\Vert u_{n+1}-x_{n+1}\right\Vert &\leq &\delta ^{2}\left[ 1-\eta
_{n}^{1}\eta _{n}^{2}\left( 1-\delta \right) \right] \left\Vert
u_{n}-x_{n}\right\Vert +\left( 1+\eta _{n}^{0}\right) \left\Vert
v_{n}-Tv_{n}\right\Vert  \notag \\
&\leq &\left[ 1-\eta _{n}^{1}\eta _{n}^{2}\left( 1-\delta \right) \right]
\left\Vert u_{n}-x_{n}\right\Vert +\left( 1+\eta _{n}^{0}\right) \left\Vert
v_{n}-Tv_{n}\right\Vert \text{.}  \label{eqn37}
\end{eqnarray}%
Denote that%
\begin{equation}
\beta _{n}:=\left\Vert u_{n}-x_{n}\right\Vert \text{, }\lambda _{n}:=\eta
_{n}^{1}\eta _{n}^{2}\left( 1-\delta \right) \in \left( 0,1\right) \text{, }%
\rho _{n}:=\left( 1+\eta _{n}^{0}\right) \left\Vert v_{n}-Tv_{n}\right\Vert 
\text{.}  \label{eqn38}
\end{equation}%
Since $\lim_{n\rightarrow \infty }\left\Vert u_{n}-x_{\ast }\right\Vert =0$
and $Tx_{\ast }=x_{\ast }$, $\lim_{n\rightarrow \infty }\left\Vert
v_{n}-Tv_{n}\right\Vert =0$ which implies $\frac{\rho _{n}}{\lambda _{n}}%
\rightarrow 0$ as $n\rightarrow \infty $. Hence an application of Lemma 1 to
(2.19) yields $\lim_{n\rightarrow \infty }\left\Vert u_{n}-x_{n}\right\Vert
=0$. Since%
\begin{equation}
\left\Vert x_{n}-x_{\ast }\right\Vert \leq \left\Vert u_{n}-x_{n}\right\Vert
+\left\Vert u_{n}-x_{\ast }\right\Vert \text{,}  \label{eqn39}
\end{equation}%
$\lim_{n\rightarrow \infty }\left\Vert x_{n}-x_{\ast }\right\Vert =0$.
\end{proof}

\begin{theorem}
Let $D$, $B$ and $T$ with fixed point $x_{\ast }$ be as in Theorem 1. Let $%
\left\{ \eta _{n}^{i}\right\} _{n=0}^{\infty }$, $i\in \left\{ 0,1,2\right\} 
$ be real sequences in $\left[ 0,1\right] $ satisfying

(i) $\eta _{0}\leq \eta _{n}^{0}\leq 1$, $\eta _{1}\leq \eta _{n}^{1}\leq 1$
and $\eta _{2}\leq \eta _{n}^{2}\leq 1$, for all $n\in 
\mathbb{N}
$ and for some $\eta _{0}$, $\eta _{1}$, $\eta _{2}>0$.

For given $x_{0}=u_{0}\in D$, consider iterative sequences $\left\{
u_{n}\right\} _{n=0}^{\infty }$ and $\left\{ x_{n}\right\} _{n=0}^{\infty }$
defined by (1.10) and (1.14), respectively. Then $\left\{ x_{n}\right\}
_{n=0}^{\infty }$ converges to $x_{\ast }$ faster than $\left\{
u_{n}\right\} _{n=0}^{\infty }$ does.
\end{theorem}

\begin{proof}
The following inequality comes from inequality (2.5) of Theorem 1%
\begin{equation}
\left\Vert x_{n+1}-x_{\ast }\right\Vert \leq \left\Vert x_{0}-x_{\ast
}\right\Vert \delta ^{2\left( n+1\right) }\prod\limits_{k=0}^{n}\left\{
1-\eta _{k}^{1}\eta _{k}^{2}\left( 1-\delta \right) \right\} \text{.}
\label{eqn40}
\end{equation}%
The following inequality is due to (\cite{K.kaya}, inequality (2.26) of
Theorem 3)%
\begin{equation}
\left\Vert u_{n+1}-x_{\ast }\right\Vert \leq \left\Vert u_{0}-x_{\ast
}\right\Vert \prod\limits_{k=0}^{n}\left[ 1-\eta _{k}^{0}\left( 1-\delta
\right) \right] \left[ 1-\eta _{n}^{1}\eta _{n}^{2}\left( 1-\delta \right) %
\right] \delta \text{,}  \label{eqn41}
\end{equation}%
which is obtained from (1.10).

An application of assumption (i) to (2.22) and (2.23), respectively, leads to%
\begin{eqnarray}
\left\Vert x_{n+1}-x_{\ast }\right\Vert &\leq &\left\Vert x_{0}-x_{\ast
}\right\Vert \delta ^{2\left( n+1\right) }\prod\limits_{k=0}^{n}\left\{
1-\eta _{1}\eta _{2}\left( 1-\delta \right) \right\}  \notag \\
&=&\left\Vert x_{0}-x_{\ast }\right\Vert \delta ^{2\left( n+1\right) }\left[
1-\eta _{1}\eta _{2}\left( 1-\delta \right) \right] ^{n+1}\text{,}
\label{eqn42}
\end{eqnarray}%
\begin{eqnarray}
\left\Vert u_{n+1}-x_{\ast }\right\Vert &\leq &\left\Vert u_{0}-x_{\ast
}\right\Vert \prod\limits_{k=0}^{n}\left[ 1-\eta _{0}\left( 1-\delta \right) %
\right] \left[ 1-\eta _{1}\eta _{2}\left( 1-\delta \right) \right] \delta 
\notag \\
&=&\left\Vert u_{0}-x_{\ast }\right\Vert \delta ^{n+1}\left[ 1-\eta
_{0}\left( 1-\delta \right) \right] ^{n+1}\left[ 1-\eta _{1}\eta _{2}\left(
1-\delta \right) \right] ^{n+1}\text{.}  \label{eqn43}
\end{eqnarray}%
Define%
\begin{equation}
a_{n}:=\delta ^{2\left( n+1\right) }\left[ 1-\eta _{1}\eta _{2}\left(
1-\delta \right) \right] ^{n+1}\left\Vert x_{0}-x_{\ast }\right\Vert \text{,}
\label{eqn44}
\end{equation}%
\begin{equation}
b_{n}:=\delta ^{n+1}\left[ 1-\eta _{0}\left( 1-\delta \right) \right] ^{n+1}%
\left[ 1-\eta _{1}\eta _{2}\left( 1-\delta \right) \right] ^{n+1}\left\Vert
u_{0}-x_{\ast }\right\Vert \text{,}  \label{eqn45}
\end{equation}%
\begin{eqnarray}
\theta _{n} &:&=\frac{a_{n}}{b_{n}}=\frac{\delta ^{2\left( n+1\right) }\left[
1-\eta _{1}\eta _{2}\left( 1-\delta \right) \right] ^{n+1}\left\Vert
x_{0}-x_{\ast }\right\Vert }{\delta ^{n+1}\left[ 1-\eta _{0}\left( 1-\delta
\right) \right] ^{n+1}\left[ 1-\eta _{1}\eta _{2}\left( 1-\delta \right) %
\right] ^{n+1}\left\Vert u_{0}-x_{\ast }\right\Vert }  \notag \\
&:&=\left[ \frac{\delta }{1-\eta _{0}\left( 1-\delta \right) }\right] ^{n+1}%
\text{.}  \label{eqn46}
\end{eqnarray}%
Since $\delta $, $\eta _{0}\in \left( 0,1\right) $%
\begin{eqnarray}
\eta _{0} &<&1  \notag \\
&\Rightarrow &\eta _{0}\left( 1-\delta \right) <1-\delta  \notag \\
&\Rightarrow &\delta <1-\eta _{0}\left( 1-\delta \right)  \notag \\
&\Rightarrow &\frac{\delta }{1-\eta _{0}\left( 1-\delta \right) }<1\text{.}
\label{eqn47}
\end{eqnarray}%
Therefore, we have%
\begin{equation}
\lim_{n\rightarrow \infty }\frac{\theta _{n+1}}{\theta _{n}}%
=\lim_{n\rightarrow \infty }\frac{\left[ \frac{\delta }{1-\eta _{0}\left(
1-\delta \right) }\right] ^{n+2}}{\left[ \frac{\delta }{1-\eta _{0}\left(
1-\delta \right) }\right] ^{n+1}}=\frac{\delta }{1-\eta _{0}\left( 1-\delta
\right) }<1\text{.}  \label{eqn48}
\end{equation}%
It thus follows from well-known ratio test that $\sum\limits_{n=0}^{\infty
}\theta _{n}<\infty $. Hence, we have $\lim_{n\rightarrow \infty }\theta
_{n}=0$ which implies that $\left\{ x_{n}\right\} _{n=0}^{\infty }$ is
faster than $\left\{ u_{n}\right\} _{n=0}^{\infty }$.
\end{proof}

In order to support analytical proof of Theorem 3 and to illustrate the
efficiency of Picard-S iteration method (1.14), we will use a numerical
example provided by Sahu \cite{Sahu} for the sake of consistent comparison.

\begin{example}
Let $B=%
\mathbb{R}
$ and $D=\left[ 0,\infty \right) $. Let $T:D\rightarrow D$ be a mapping
defined by $Tx=\sqrt[3]{3x+18}$ for all $x\in D$. $T$ is a contraction with
contractivity factor $\delta =1/\sqrt[3]{18}$ and $x_{\ast }=3$, see \cite%
{Sahu}. Take $\eta _{n}^{0}=\eta _{k}^{1}=\eta _{k}^{2}=1/2$ with initial
value $x_{0}=1000$. The following tables 1-2-3 show that Picard-S iteration
method (1.14) converges faster than all Picard (1.3), Mann (1.5), Ishikawa
(1.6), Noor (1.7), SP (1.8), S (1.9) , CR (1.10), Normal-S \cite{SahuPet},
S* \cite{Karahan} and Two-Step Mann \cite{iam} iteration methods including a
new three-step iteration method due to Abbas and Nazir \cite{Abbas}.
\end{example}

\begin{center}
\begin{tabular}{ccccc}
\multicolumn{5}{c}{\ \ Table 1. Comparison speed of convergence among
various iteration methods} \\ 
No. of Iter. & Picard-S & S & Normal-S & S* \\ 
1 & 3.848449787 & 12.99923955 & 11.54719590 & 9.149866494 \\ 
2 & 3.007911860 & 3.679603368 & 3.446604313 & 3.309032025 \\ 
3 & 3.000075950 & 3.057482809 & 3.027275056 & 3.018443229 \\ 
4 & 3.000000729 & 3.004958405 & 3.001682532 & 3.001112415 \\ 
5 & 3.000000007 & 3.000428434 & 3.000103856 & 3.000067138 \\ 
6 & 3.000000000 & 3.000037025 & 3.000006411 & 3.000004052 \\ 
7 & 3.000000000 & 3.000003200 & 3.000000396 & 3.000000245 \\ 
8 & 3.000000000 & 3.000000276 & 3.000000024 & 3.000000016 \\ 
9 & 3.000000000 & 3.000000024 & 3.000000001 & 3.000000001 \\ 
10 & 3.000000000 & 3.000000002 & 3.000000000 & 3.000000000 \\ 
11 & 3.000000000 & 3.000000000 & 3.000000000 & 3.000000000 \\ 
$\vdots $ & $\vdots $ & $\vdots $ & $\vdots $ & $\vdots $%
\end{tabular}

\begin{tabular}{ccccc}
\multicolumn{5}{c}{Table 2. Comparison speed of convergence among various
iteration methods} \\ 
No. of Iter. & Picard & SP & CR & Abbas \& Nazir \\ 
1 & 14.45128320 & 134.3273583 & 8.423844669 & 7.697822844 \\ 
2 & 3.944094141 & 21.81696908 & 3.224582695 & 3.150763238 \\ 
3 & 3.101431265 & 5.994481952 & 3.010704011 & 3.005356237 \\ 
4 & 3.011228065 & 3.504188621 & 3.000513730 & 3.000191005 \\ 
5 & 3.001247045 & 3.086158574 & 3.000024664 & 3.000006812 \\ 
6 & 3.000138554 & 3.014764667 & 3.000001184 & 3.000000242 \\ 
7 & 3.000015395 & 3.002531406 & 3.000000057 & 3.000000010 \\ 
8 & 3.000001710 & 3.000434048 & 3.000000002 & 3.000000000 \\ 
9 & 3.000000190 & 3.000074424 & 3.000000000 & 3.000000000 \\ 
10 & 3.000000021 & 3.000012761 & 3.000000000 & 3.000000000 \\ 
11 & 3.000000002 & 3.000002187 & 3.000000000 & 3.000000000 \\ 
12 & 3.000000000 & 3.000000376 & 3.000000000 & 3.000000000 \\ 
13 & 3.000000000 & 3.000000064 & 3.000000000 & 3.000000000 \\ 
14 & 3.000000000 & 3.000000011 & 3.000000000 & 3.000000000 \\ 
15 & 3.000000000 & 3.000000002 & 3.000000000 & 3.000000000 \\ 
16 & 3.000000000 & 3.000000000 & 3.000000000 & 3.000000000 \\ 
$\vdots $ & $\vdots $ & $\vdots $ & $\vdots $ & $\vdots $%
\end{tabular}

\begin{tabular}{ccccc}
\multicolumn{5}{c}{Table 3. Comparison speed of convergence among various
iteration methods} \\ 
No. of Iter. & Mann & Ishikawa & Noor & Two-step Mann \\ 
1 & 507.2256416 & 505.7735980 & 505.7681478 & 259.3864188 \\ 
2 & 259.3864188 & 257.5166864 & 257.5072981 & 70.91103158 \\ 
3 & 134.3273583 & 132.4972579 & 132.4845970 & 21.81696908 \\ 
4 & 70.91103158 & 69.30291326 & 69.28740168 & 8.474131740 \\ 
5 & 38.52222997 & 37.19141418 & 37.17367673 & 4.647951504 \\ 
6 & 21.81696908 & 20.76243714 & 20.74360422 & 3.504188621 \\ 
7 & 13.09346232 & 12.28975230 & 12.27143253 & 3.155173770 \\ 
8 & 8.474131740 & 7.884743216 & 7.868548900 & 3.047850681 \\ 
9 & 5.994481952 & 5.578322215 & 5.565269205 & 3.014764667 \\ 
10 & 4.647951504 & 4.364258381 & 4.354541966 & 3.004556608 \\ 
$\vdots $ & $\vdots $ & $\vdots $ & $\vdots $ & $\vdots $ \\ 
20 & 3.004556608 & 3.002415570 & 3.002338681 & 3.000000036 \\ 
21 & 3.002531406 & 3.001282335 & 3.001238310 & 3.000000011 \\ 
22 & 3.001406323 & 3.000680745 & 3.000655675 & 3.000000003 \\ 
23 & 3.000781287 & 3.000361382 & 3.000347175 & 3.000000001 \\ 
24 & 3.000434048 & 3.000191845 & 3.000183827 & 3.000000000 \\ 
$\vdots $ & $\vdots $ & $\vdots $ & $\vdots $ & $\vdots $ \\ 
41 & 3.000000020 & 3.000000004 & 3.000000003 & 3.000000000 \\ 
42 & 3.000000011 & 3.000000002 & 3.000000002 & 3.000000000 \\ 
43 & 3.000000006 & 3.000000001 & 3.000000001 & 3.000000000 \\ 
44 & 3.000000003 & 3.000000000 & 3.000000000 & 3.000000000 \\ 
45 & 3.000000002 & 3.000000000 & 3.000000000 & 3.000000000 \\ 
46 & 3.000000001 & 3.000000000 & 3.000000000 & 3.000000000 \\ 
47 & 3.000000000 & 3.000000000 & 3.000000000 & 3.000000000 \\ 
$\vdots $ & $\vdots $ & $\vdots $ & $\vdots $ & $\vdots $%
\end{tabular}
\end{center}

\textbf{A Data Dependence Result}

We are now able to establish the following data dependence result.

\begin{theorem}
Let $\widetilde{T}$ be an approximate operator of $T$ satisfying condition
(1.4). Let $\left\{ x_{n}\right\} _{n=0}^{\infty }$ be an iterative sequence
generated by (1.14) for $T$ and define an iterative sequence $\left\{ 
\widetilde{x}_{n}\right\} _{n=0}^{\infty }$ as follows%
\begin{equation}
\left\{ 
\begin{array}{c}
\widetilde{x}_{0}\in D\text{, \ \ \ \ \ \ \ \ \ \ \ \ \ \ \ \ \ \ \ \ \ \ \
\ \ \ \ \ \ \ \ \ \ \ \ \ \ \ \ \ } \\ 
\widetilde{x}_{n+1}=\widetilde{T}\widetilde{y}_{n}\text{, \ \ \ \ \ \ \ \ \
\ \ \ \ \ \ \ \ \ \ \ \ \ \ \ \ \ \ \ \ \ \ \ \ \ } \\ 
\widetilde{y}_{n}=\left( 1-\eta _{n}^{1}\right) \widetilde{T}\widetilde{x}%
_{n}+\eta _{n}^{1}\widetilde{T}\widetilde{z}_{n}\text{, \ \ \ \ \ \ } \\ 
\widetilde{z}_{n}=\left( 1-\eta _{n}^{2}\right) \widetilde{x}_{n}+\eta
_{n}^{2}\widetilde{T}\widetilde{x}_{n}\text{, }n\in 
\mathbb{N}
\text{,}%
\end{array}%
\right.  \label{eqn49}
\end{equation}%
where $\left\{ \eta _{n}^{i}\right\} _{n=0}^{\infty }$, $i\in \left\{
0,1,2\right\} $ be real sequences in $\left[ 0,1\right] $ satisfying (i) $%
\frac{1}{2}\leq \eta _{n}^{1}\eta _{n}^{2}$ for all $n\in 
\mathbb{N}
$, and (ii) $\sum\limits_{n=0}^{\infty }\eta _{n}^{1}\eta _{n}^{2}=\infty $.
If $Tx_{\ast }=x_{\ast }$ and $\widetilde{T}\widetilde{x}_{\ast }=\widetilde{%
x}_{\ast }$ such that $\widetilde{x}_{n}\rightarrow \widetilde{x}_{\ast }$
as $n\rightarrow \infty $, then we have%
\begin{equation}
\left\Vert x_{\ast }-\widetilde{x}_{\ast }\right\Vert \leq \frac{%
5\varepsilon }{1-\delta }\text{,}  \label{eqn50}
\end{equation}%
where $\varepsilon >0$ is a fixed number.
\end{theorem}

\begin{proof}
It follows from (1.4), (1.14), and (2.31) that%
\begin{eqnarray}
\left\Vert x_{n+1}-\widetilde{x}_{n+1}\right\Vert &=&\left\Vert Ty_{n}-T%
\widetilde{y}_{n}+T\widetilde{y}_{n}-\widetilde{T}\widetilde{y}%
_{n}\right\Vert  \notag \\
&\leq &\left\Vert Ty_{n}-T\widetilde{y}_{n}\right\Vert +\left\Vert T%
\widetilde{y}_{n}-\widetilde{T}\widetilde{y}_{n}\right\Vert  \notag \\
&\leq &\delta \left\Vert y_{n}-\widetilde{y}_{n}\right\Vert +\varepsilon 
\text{,}  \label{eqn51}
\end{eqnarray}%
\begin{eqnarray}
\left\Vert y_{n}-\widetilde{y}_{n}\right\Vert &\leq &\left( 1-\eta
_{n}^{1}\right) \left\Vert Tx_{n}-\widetilde{T}\widetilde{x}_{n}\right\Vert
+\eta _{n}^{1}\left\Vert Tz_{n}-\widetilde{T}\widetilde{z}_{n}\right\Vert 
\notag \\
&\leq &\left( 1-\eta _{n}^{1}\right) \left\{ \left\Vert Tx_{n}-T\widetilde{x}%
_{n}\right\Vert +\left\Vert T\widetilde{x}_{n}-\widetilde{T}\widetilde{x}%
_{n}\right\Vert \right\}  \notag \\
&&+\eta _{n}^{1}\left\{ \left\Vert Tz_{n}-T\widetilde{z}_{n}\right\Vert
+\left\Vert T\widetilde{z}_{n}-\widetilde{T}\widetilde{z}_{n}\right\Vert
\right\}  \notag \\
&\leq &\left( 1-\eta _{n}^{1}\right) \left\{ \delta \left\Vert x_{n}-%
\widetilde{x}_{n}\right\Vert +\varepsilon \right\} +\eta _{n}^{1}\left\{
\delta \left\Vert z_{n}-\widetilde{z}_{n}\right\Vert +\varepsilon \right\} 
\text{,}  \label{eqn52}
\end{eqnarray}%
\begin{eqnarray}
\left\Vert z_{n}-\widetilde{z}_{n}\right\Vert &\leq &\left( 1-\eta
_{n}^{2}\right) \left\Vert x_{n}-\widetilde{x}_{n}\right\Vert +\eta
_{n}^{2}\left\Vert Tx_{n}-\widetilde{T}\widetilde{x}_{n}\right\Vert  \notag
\\
&\leq &\left( 1-\eta _{n}^{2}\right) \left\Vert x_{n}-\widetilde{x}%
_{n}\right\Vert +\eta _{n}^{2}\left\{ \left\Vert Tx_{n}-T\widetilde{x}%
_{n}\right\Vert +\left\Vert T\widetilde{x}_{n}-\widetilde{T}\widetilde{x}%
_{n}\right\Vert \right\}  \notag \\
&\leq &\left[ 1-\eta _{n}^{2}\left( 1-\delta \right) \right] \left\Vert
x_{n}-\widetilde{x}_{n}\right\Vert +\eta _{n}^{2}\varepsilon \text{.}
\label{eqn53}
\end{eqnarray}%
Combining (2.33), (2.34), and (2.35) and using the facts $\delta $ and $%
\delta ^{2}\in \left( 0,1\right) $%
\begin{eqnarray*}
\left\Vert x_{n+1}-\widetilde{x}_{n+1}\right\Vert &\leq &\delta ^{2}\left[
1-\eta _{n}^{1}\eta _{n}^{2}\left( 1-\delta \right) \right] \left\Vert x_{n}-%
\widetilde{x}_{n}\right\Vert +\eta _{n}^{1}\delta ^{2}\eta
_{n}^{2}\varepsilon +\left( 1-\eta _{n}^{1}\right) \delta \varepsilon
+\delta \eta _{n}^{1}\varepsilon +\varepsilon \\
&\leq &\left[ 1-\eta _{n}^{1}\eta _{n}^{2}\left( 1-\delta \right) \right]
\left\Vert x_{n}-\widetilde{x}_{n}\right\Vert +\eta _{n}^{1}\eta
_{n}^{2}\varepsilon +2\varepsilon \text{.}
\end{eqnarray*}%
From assumption (i) we have%
\begin{equation*}
1-\eta _{n}^{1}\eta _{n}^{2}\leq \eta _{n}^{1}\eta _{n}^{2}\text{,}
\end{equation*}%
and thus, inequality (2.36) becomes%
\begin{eqnarray*}
\left\Vert x_{n+1}-\widetilde{x}_{n+1}\right\Vert &\leq &\left[ 1-\eta
_{n}^{1}\eta _{n}^{2}\left( 1-\delta \right) \right] \left\Vert x_{n}-%
\widetilde{x}_{n}\right\Vert +\eta _{n}^{1}\eta _{n}^{2}\varepsilon +2\left(
1-\eta _{n}^{1}\eta _{n}^{2}+\eta _{n}^{1}\eta _{n}^{2}\right) \varepsilon \\
&\leq &\left[ 1-\eta _{n}^{1}\eta _{n}^{2}\left( 1-\delta \right) \right]
\left\Vert x_{n}-\widetilde{x}_{n}\right\Vert +\eta _{n}^{1}\eta
_{n}^{2}\left( 1-\delta \right) \frac{5\varepsilon }{1-\delta }\text{.}
\end{eqnarray*}%
Denote that%
\begin{equation*}
\beta _{n}:=\left\Vert x_{n}-\widetilde{x}_{n}\right\Vert \text{, }\mu
_{n}:=\eta _{n}^{1}\eta _{n}^{2}\left( 1-\delta \right) \in \left(
0,1\right) \text{, }\gamma _{n}:=\frac{5\varepsilon }{1-\delta }\text{. }
\end{equation*}%
It follows from Lemma 2 that%
\begin{equation}
0\leq \underset{n\rightarrow \infty }{\lim \sup }\left\Vert x_{n}-\widetilde{%
x}_{n}\right\Vert \leq \underset{n\rightarrow \infty }{\lim \sup }\frac{%
5\varepsilon }{1-\delta }\text{.}  \label{eqn58}
\end{equation}%
From Theorem 1 we know that $\lim_{n\rightarrow \infty }x_{n}=x_{\ast }$.
Thus, using this fact together with the assumption $\lim_{n\rightarrow
\infty }\widetilde{x}_{n}=\widetilde{x}_{\ast }$ we obtain%
\begin{equation}
\left\Vert x_{\ast }-\widetilde{x}_{\ast }\right\Vert \leq \frac{%
5\varepsilon }{1-\delta }\text{.}  \label{eqn59}
\end{equation}
\end{proof}

\textbf{An Application}

Throughout the rest of this paper, the space $C\left( \left[ a,b\right]
\right) $ with endowed Chebyshev norm $\left\Vert x-y\right\Vert _{\infty }=$
$\max_{t\in \left[ a,b\right] }\left\vert x\left( t\right) -y\left( t\right)
\right\vert $ denotes the space of all continuous real-valued functions on a
closed interval $[a,b]$. It is well-known that $\left( C\left( \left[ a,b%
\right] \right) \text{,}\left\Vert \cdot \right\Vert _{\infty }\right) $ is
a Banach space, see \cite{Hammerlin}.

In this section we will be interested in the following delay differential
equation%
\begin{equation}
x^{\prime }\left( t\right) =f\left( t,x\left( t\right) ,x\left( t-\tau
\right) \right) \text{, }t\in \left[ t_{0},b\right]  \label{eqn60}
\end{equation}%
with initial condition%
\begin{equation}
x\left( t\right) =\psi \left( t\right) \text{,\ }t\in \left[ t_{0}-\tau
,t_{0}\right] \text{.}  \label{eqn61}
\end{equation}%
We opine that the following conditions are performed

\textbf{(A}$_{1}$\textbf{) }$t_{0}$,$b\in 
\mathbb{R}
$,$\tau >0$;

\textbf{(A}$_{2}$\textbf{) }$f\in C\left( \left[ t_{0},b\right] \times 
\mathbb{R}
^{2}\text{,}%
\mathbb{R}
\right) $;

\textbf{(A}$_{3}$\textbf{) }$\psi \in C\left( \left[ t_{0}-\tau ,b\right] 
\text{,}%
\mathbb{R}
\right) $;

\textbf{(A}$_{4}$\textbf{) }there exist $L_{f}>0$ such that%
\begin{equation}
\left\vert f\left( t,u_{1},u_{2}\right) -f\left( t,v_{1},v_{2}\right)
\right\vert \leq L_{f}\sum_{i=1}^{2}\left\vert u_{i}-v_{i}\right\vert \text{%
, }\forall u_{i}\text{,}v_{i}\in 
\mathbb{R}
\text{, }i=1,2\text{, }t\in \left[ t_{0},b\right] \text{;}  \label{eqn62}
\end{equation}

\textbf{(A}$_{5}$\textbf{) }$2L_{f}\left( b-t_{0}\right) <1$.

By a solution of the problem (2.38)-(2.39) we understand a function $x\in
C\left( \left[ t_{0}-\tau ,b\right] ,%
\mathbb{R}
\right) \cap C^{1}\left( \left[ t_{0},b\right] ,%
\mathbb{R}
\right) $.

The problem (2.38)-(2.39) can be reformulated in the following form of
integral equation%
\begin{equation}
x\left( t\right) =\left\{ 
\begin{array}{c}
\psi \left( t\right) \text{,\ \ }t\in \left[ t_{0}-\tau ,t_{0}\right] \text{%
, \ \ \ \ \ \ \ \ \ \ \ \ \ \ \ \ \ \ \ \ \ \ \ \ \ \ \ \ \ \ \ } \\ 
\psi \left( t_{0}\right) +\dint\nolimits_{t_{0}}^{t}f\left( s,x\left(
s\right) ,x\left( s-\tau \right) \right) ds\text{, \ }t\in \left[ t_{0},b%
\right] \text{.}%
\end{array}%
\right.  \label{eqn63}
\end{equation}%
The following result can be found in \cite{coman}.

\begin{theorem}
Assume that conditions $(A_{1})-(A_{5})$ are hold. Then the problem $%
(2.38)-(2.39)$ has a unique solution, say $x_{\ast }$, in $C\left( \left[
t_{0}-\tau ,b\right] ,%
\mathbb{R}
\right) \cap C^{1}\left( \left[ t_{0},b\right] ,%
\mathbb{R}
\right) $ and%
\begin{equation*}
x_{\ast }=\lim_{n\rightarrow \infty }T^{n}\left( x\right) \text{ for any }%
x\in C\left( \left[ t_{0}-\tau ,b\right] ,%
\mathbb{R}
\right) \text{.}
\end{equation*}
\end{theorem}

Now we are in a position to give the following result.

\begin{theorem}
Suppose that conditions $(A_{1})-(A_{5})$ are satisfied. Then the problem $%
(2.38)-(2.39)$ has a unique solution, say $x_{\ast }$, in $C\left( \left[
t_{0}-\tau ,b\right] ,%
\mathbb{R}
\right) \cap C^{1}\left( \left[ t_{0},b\right] ,%
\mathbb{R}
\right) $ and Picard-S iteration method (1.14) converges to $x_{\ast }$.
\end{theorem}

\begin{proof}
Let $\left\{ x_{n}\right\} _{n=0}^{\infty }$ be a iterative sequence
generated by Picard-S iteration method (1.14) for the operator%
\begin{equation}
Tx\left( t\right) =\left\{ 
\begin{array}{c}
\psi \left( t\right) \text{,\ \ }t\in \left[ t_{0}-\tau ,t_{0}\right] \text{%
, \ \ \ \ \ \ \ \ \ \ \ \ \ \ \ \ \ \ \ \ \ \ \ \ \ \ \ \ \ \ \ } \\ 
\psi \left( t_{0}\right) +\dint\nolimits_{t_{0}}^{t}f\left( s,x\left(
s\right) ,x\left( s-\tau \right) \right) ds\text{, \ }t\in \left[ t_{0},b%
\right] \text{.}%
\end{array}%
\right.  \label{eqn64}
\end{equation}%
Denote by $x_{\ast }$ the fixed point of $T$. We will show that $%
x_{n}\rightarrow x_{\ast }$ as $n\rightarrow \infty $.

For $t\in \left[ t_{0}-\tau ,t_{0}\right] $, it is easy to see that $%
x_{n}\rightarrow x_{\ast }$ as $n\rightarrow \infty $.

For $t\in \left[ t_{0},b\right] $ we obtain%
\begin{eqnarray}
\left\Vert z_{n}-x_{\ast }\right\Vert _{\infty } &=&\left\Vert \left( 1-\eta
_{n}^{2}\right) x_{n}+\eta _{n}^{2}Tx_{n}-Tx_{\ast }\right\Vert _{\infty } 
\notag \\
&\leq &\left( 1-\eta _{n}^{2}\right) \left\Vert x_{n}-x_{\ast }\right\Vert
_{\infty }+\eta _{n}^{2}\left\Vert Tx_{n}-Tx_{\ast }\right\Vert _{\infty } 
\notag \\
&=&\left( 1-\eta _{n}^{2}\right) \left\Vert x_{n}-x_{\ast }\right\Vert
_{\infty }+\eta _{n}^{2}\max_{t\in \left[ t_{0}-\tau ,b\right] }\left\vert
Tx_{n}\left( t\right) -Tx_{\ast }\left( t\right) \right\vert  \notag \\
&=&\left( 1-\eta _{n}^{2}\right) \left\Vert x_{n}-x_{\ast }\right\Vert
_{\infty }  \notag \\
&&+\eta _{n}^{2}\max_{t\in \left[ t_{0}-\tau ,b\right] }\left\vert \psi
\left( t_{0}\right) +\dint\nolimits_{t_{0}}^{t}f\left( s,x_{n}\left(
s\right) ,x_{n}\left( s-\tau \right) \right) ds\right.  \notag \\
&&\left. -\psi \left( t_{0}\right) -\dint\nolimits_{t_{0}}^{t}f\left(
s,x_{\ast }\left( s\right) ,x_{\ast }\left( s-\tau \right) \right)
ds\right\vert  \notag \\
&=&\left( 1-\eta _{n}^{2}\right) \left\Vert x_{n}-x_{\ast }\right\Vert
_{\infty }  \notag \\
&&+\eta _{n}^{2}\max_{t\in \left[ t_{0}-\tau ,b\right] }\left\vert
\dint\nolimits_{t_{0}}^{t}f\left( s,x_{n}\left( s\right) ,x_{n}\left( s-\tau
\right) \right) ds\right.  \notag \\
&&\left. -\dint\nolimits_{t_{0}}^{t}f\left( s,x_{\ast }\left( s\right)
,x_{\ast }\left( s-\tau \right) \right) ds\right\vert  \notag \\
&\leq &\left( 1-\eta _{n}^{2}\right) \left\Vert x_{n}-x_{\ast }\right\Vert
_{\infty }  \notag \\
&&+\eta _{n}^{2}\max_{t\in \left[ t_{0}-\tau ,b\right] }\dint%
\nolimits_{t_{0}}^{t}\left\vert f\left( s,x_{n}\left( s\right) ,x_{n}\left(
s-\tau \right) \right) -f\left( s,x_{\ast }\left( s\right) ,x_{\ast }\left(
s-\tau \right) \right) \right\vert ds  \notag \\
&\leq &\left( 1-\eta _{n}^{2}\right) \left\Vert x_{n}-x_{\ast }\right\Vert
_{\infty }  \notag \\
&&+\eta _{n}^{2}\max_{t\in \left[ t_{0}-\tau ,b\right] }\dint%
\nolimits_{t_{0}}^{t}L_{f}\left( \left\vert x_{n}\left( s\right) -x_{\ast
}\left( s\right) \right\vert +\left\vert x_{n}\left( s-\tau \right) -x_{\ast
}\left( s-\tau \right) \right\vert \right) ds  \notag \\
&\leq &\left( 1-\eta _{n}^{2}\right) \left\Vert x_{n}-x_{\ast }\right\Vert
_{\infty }  \notag \\
&&+\eta _{n}^{2}\dint\nolimits_{t_{0}}^{t}L_{f}\left( \max_{s\in \left[
t_{0}-\tau ,b\right] }\left\vert x_{n}\left( s\right) -x_{\ast }\left(
s\right) \right\vert +\max_{s\in \left[ t_{0}-\tau ,b\right] }\left\vert
x_{n}\left( s-\tau \right) -x_{\ast }\left( s-\tau \right) \right\vert
\right) ds  \notag \\
&\leq &\left( 1-\eta _{n}^{2}\right) \left\Vert x_{n}-x_{\ast }\right\Vert
_{\infty }  \notag \\
&&+\eta _{n}^{2}\dint\nolimits_{t_{0}}^{t}L_{f}\left( \left\Vert
x_{n}-x_{\ast }\right\Vert _{\infty }+\left\Vert x_{n}-x_{\ast }\right\Vert
_{\infty }\right) ds  \notag \\
&\leq &\left( 1-\eta _{n}^{2}\right) \left\Vert x_{n}-x_{\ast }\right\Vert
_{\infty }+2\eta _{n}^{2}L_{f}\left( t-t_{0}\right) \left\Vert x_{n}-x_{\ast
}\right\Vert _{\infty }  \notag \\
&\leq &\left[ 1-\eta _{n}^{2}\left( 1-2L_{f}\left( b-t_{0}\right) \right) %
\right] \left\Vert x_{n}-x_{\ast }\right\Vert _{\infty }\text{,}
\label{eqn65}
\end{eqnarray}%
\begin{eqnarray}
\left\Vert y_{n}-x_{\ast }\right\Vert _{\infty } &\leq &\left( 1-\eta
_{n}^{1}\right) \left\Vert Tx_{n}-Tx_{\ast }\right\Vert _{\infty }+\eta
_{n}^{1}\left\Vert Tz_{n}-Tx_{\ast }\right\Vert _{\infty }  \notag \\
&=&\left( 1-\eta _{n}^{1}\right) \max_{t\in \left[ t_{0}-\tau ,b\right]
}\left\vert \dint\nolimits_{t_{0}}^{t}\left[ f\left( s,x_{n}\left( s\right)
,x_{n}\left( s-\tau \right) \right) -f\left( s,x_{\ast }\left( s\right)
,x_{\ast }\left( s-\tau \right) \right) \right] ds\right\vert  \notag \\
&&+\eta _{n}^{1}\max_{t\in \left[ t_{0}-\tau ,b\right] }\left\vert
\dint\nolimits_{t_{0}}^{t}\left[ f\left( s,z_{n}\left( s\right) ,z_{n}\left(
s-\tau \right) \right) -f\left( s,x_{\ast }\left( s\right) ,x_{\ast }\left(
s-\tau \right) \right) \right] ds\right\vert  \notag \\
&\leq &\left( 1-\eta _{n}^{1}\right) \max_{t\in \left[ t_{0}-\tau ,b\right]
}\dint\nolimits_{t_{0}}^{t}\left\vert f\left( s,x_{n}\left( s\right)
,x_{n}\left( s-\tau \right) \right) -f\left( s,x_{\ast }\left( s\right)
,x_{\ast }\left( s-\tau \right) \right) \right\vert ds  \notag \\
&&+\eta _{n}^{1}\max_{t\in \left[ t_{0}-\tau ,b\right] }\dint%
\nolimits_{t_{0}}^{t}\left\vert f\left( s,z_{n}\left( s\right) ,z_{n}\left(
s-\tau \right) \right) -f\left( s,x_{\ast }\left( s\right) ,x_{\ast }\left(
s-\tau \right) \right) \right\vert ds  \notag \\
&\leq &\left( 1-\eta _{n}^{1}\right) \max_{t\in \left[ t_{0}-\tau ,b\right]
}\dint\nolimits_{t_{0}}^{t}L_{f}\left( \left\vert x_{n}\left( s\right)
-x_{\ast }\left( s\right) \right\vert +\left\vert x_{n}\left( s-\tau \right)
-x_{\ast }\left( s-\tau \right) \right\vert \right) ds  \notag \\
&&+\eta _{n}^{1}\max_{t\in \left[ t_{0}-\tau ,b\right] }\dint%
\nolimits_{t_{0}}^{t}L_{f}\left( \left\vert z_{n}\left( s\right) -x_{\ast
}\left( s\right) \right\vert +\left\vert z_{n}\left( s-\tau \right) -x_{\ast
}\left( s-\tau \right) \right\vert \right) ds  \notag \\
&\leq &2L_{f}\left( b-t_{0}\right) \left\{ \left( 1-\eta _{n}^{1}\right)
\left\Vert x_{n}-x_{\ast }\right\Vert _{\infty }+\eta _{n}^{1}\left\Vert
z_{n}-x_{\ast }\right\Vert _{\infty }\right\} \text{,}  \label{eqn66}
\end{eqnarray}%
\begin{eqnarray}
\left\Vert x_{n+1}-x_{\ast }\right\Vert _{\infty } &=&\left\Vert
Ty_{n}-Tx_{\ast }\right\Vert _{\infty }  \notag \\
&=&\max_{t\in \left[ t_{0}-\tau ,b\right] }\left\vert
\dint\nolimits_{t_{0}}^{t}\left[ f\left( s,y_{n}\left( s\right) ,y_{n}\left(
s-\tau \right) \right) -f\left( s,x_{\ast }\left( s\right) ,x_{\ast }\left(
s-\tau \right) \right) \right] ds\right\vert  \notag \\
&\leq &\max_{t\in \left[ t_{0}-\tau ,b\right] }\dint\nolimits_{t_{0}}^{t}%
\left\vert f\left( s,y_{n}\left( s\right) ,y_{n}\left( s-\tau \right)
\right) -f\left( s,x_{\ast }\left( s\right) ,x_{\ast }\left( s-\tau \right)
\right) \right\vert ds  \notag \\
&\leq &\max_{t\in \left[ t_{0}-\tau ,b\right] }\dint%
\nolimits_{t_{0}}^{t}L_{f}\left( \left\vert y_{n}\left( s\right) -x_{\ast
}\left( s\right) \right\vert +\left\vert y_{n}\left( s-\tau \right) -x_{\ast
}\left( s-\tau \right) \right\vert \right) ds  \notag \\
&\leq &2L_{f}\left( b-t_{0}\right) \left\Vert y_{n}-x_{\ast }\right\Vert
_{\infty }\text{.}  \label{eqn67}
\end{eqnarray}%
Combining (2.43), (2.44), and (2.45)%
\begin{equation}
\left\Vert x_{n+1}-x_{\ast }\right\Vert _{\infty }\leq 4L_{f}^{2}\left(
b-t_{0}\right) ^{2}\left[ 1-\eta _{n}^{1}\eta _{n}^{2}\left( 1-2L_{f}\left(
b-t_{0}\right) \right) \right] \left\Vert x_{n}-x_{\ast }\right\Vert
_{\infty }\text{,}  \label{eqn68}
\end{equation}%
or from assumption $(A_{5})$%
\begin{equation}
\left\Vert x_{n+1}-x_{\ast }\right\Vert _{\infty }\leq \left[ 1-\eta
_{n}^{1}\eta _{n}^{2}\left( 1-2L_{f}\left( b-t_{0}\right) \right) \right]
\left\Vert x_{n}-x_{\ast }\right\Vert _{\infty }\text{.}  \label{eqn69}
\end{equation}%
Therefore, inductively%
\begin{equation}
\left\Vert x_{n+1}-x_{\ast }\right\Vert _{\infty }\leq
\dprod\limits_{k=0}^{n}\left[ 1-\eta _{k}^{1}\eta _{k}^{2}\left(
1-2L_{f}\left( b-t_{0}\right) \right) \right] \left\Vert x_{0}-x_{\ast
}\right\Vert _{\infty }\text{.}  \label{eqn70}
\end{equation}%
Since $\eta _{n}^{i}\in \left[ 0,1\right] $, for all $n\in 
\mathbb{N}
$ and for each $i\in \left\{ 1,2\right\} $, assumption $(A_{5})$ yields%
\begin{equation}
1-\eta _{n}^{1}\eta _{n}^{2}\left( 1-2L_{f}\left( b-t_{0}\right) \right) <1%
\text{.}  \label{eqn71}
\end{equation}%
Utilizing the same argument as in the proof of Theorem 1, we obtain%
\begin{equation}
\left\Vert x_{n+1}-x_{\ast }\right\Vert _{\infty }\leq \left\Vert
x_{0}-x_{\ast }\right\Vert _{\infty }^{n+1}e^{-\left( 1-2L_{f}\left(
b-t_{0}\right) \right) \sum_{k=0}^{n}\eta _{k}^{1}\eta _{k}^{2}}\text{,}
\label{eqn72}
\end{equation}%
which lead us to $\lim_{n\rightarrow \infty }\left\Vert x_{n}-x_{\ast
}\right\Vert _{\infty }=0$.
\end{proof}

\end{document}